\documentclass[12pt]{amsart}

\usepackage{ucs}

\usepackage{amssymb}
\usepackage{amsthm}
\usepackage{amsmath}
\usepackage{latexsym}
\usepackage[cp1251]{inputenc}
\usepackage{graphicx}
\usepackage{wrapfig}
\usepackage{caption}
\usepackage{subcaption}
\usepackage{indentfirst}
\usepackage[left=2.6cm,right=2.6cm,top=2.6cm,bottom=2.6cm,bindingoffset=0cm]{geometry}
\usepackage{enumerate}
\usepackage{makecell}

\DeclareMathOperator{\aut}{Aut}

\DeclareMathOperator{\cay}{Cay}
\DeclareMathOperator{\cyc}{Cyc}

\DeclareMathOperator{\GL}{GL}

\DeclareMathOperator{\Inn}{Inn}

\DeclareMathOperator{\orb}{Orb}

\DeclareMathOperator{\rk}{rk}

\DeclareMathOperator{\Span}{Span}

\DeclareMathOperator{\sym}{Sym}
\DeclareMathOperator{\rad}{rad}

\DeclareMathOperator{\GCD}{GCD}
\DeclareMathOperator{\Cla}{Cla}

\makeatletter 
\def\@seccntformat#1{\csname the#1\endcsname. } 
\def\@biblabel#1{#1.}

\makeatother

\title{On generalized Schur groups}

\author{Grigory Ryabov}
\address{Sobolev Institute of Mathematics, Novosibirsk, Russia}
\address{Novosibirsk State Technical University, Novosibirsk, Russia}
\email{gric2ryabov@gmail.com}

\thanks{The work is supported by Russian Scientific Fund (project No.~22-71-00021)}

\date{}

\newtheorem{prop}{Proposition}[section]

\newtheorem{lemm}[prop]{Lemma}
\newtheorem{theo}[prop]{Theorem}

\newtheorem{corl}{Corollary}[section]

\theoremstyle{definition}

\newtheorem*{rem1}{Remark 1}
\newtheorem*{rem2}{Remark 2}
\newtheorem*{prob}{Problem}

\begin{document}

\vspace{\baselineskip}
\vspace{\baselineskip}

\vspace{\baselineskip}

\vspace{\baselineskip}

\begin{abstract}
An $S$-ring (Schur ring) is called \emph{central} if it is contained in the center of the group ring. We introduce the notion of a \emph{generalized Schur group}, i.e. such finite group that all central $S$-rings over this group are schurian. It generalizes in a natural way the notion of a \emph{Schur group} and they are equivalent for abelian groups. We establish basic properties and provide infinite families of nonabelian generalized Schur groups.
\end{abstract}

\maketitle

\section{Introduction}

Let $G$ be a finite group. A subring of the group ring $\mathbb{Z}G$ is called an \emph{$S$-ring} (a \emph{Schur ring}) over $G$ if it is a free $\mathbb{Z}$-module spanned on a special partition of $G$ (exact definitions are given in Section~$2$). The notion of  $S$-ring goes back to Schur and Wielandt. They used ``the S-ring method'' to study a permutation group having a regular subgroup~\cite{Schur,Wi}. The $S$-rings is a powerful tool for studying problems of the algebraic combinatorics, in particular, the Cayley graph isomorphism problem. For more details, we refer the readers to the survey~\cite{MP}.

In~\cite{Hen}, it was proved that $S$-rings are closely related to \emph{supercharacters} introduced in~\cite{DI} for studying representations of algebraic groups. Namely, the supercharacters of the group~$G$ are in one-to-one correspondence with the $S$-rings contained in the center $Z(\mathbb{Z}G)$ of the group ring (see~\cite[Proposition~2.4]{Hen}). Following~\cite{CMP}, we call such $S$-rings \emph{central}. Observe that if $G$ is abelian, then $Z(\mathbb{Z}G)=\mathbb{Z}G$ and hence every $S$-ring over $G$ is central. 

The most of the results in the theory of $S$-rings are concerned with $S$-rings over abelian groups, whereas the general theory of $S$-rings over nonabelian groups remains almost undeveloped. One of the possible ways to study $S$-rings over nonabelian groups and extend the results from abelian case to the general one is to restrict ourselves to central $S$-rings. The classical Schur theorem on multipliers for $S$-rings over an abelian group~\cite[Theorem~23.9]{Wi} was extended to central $S$-rings over an arbitrary group in~\cite[Theorem~1.2]{CMP}. In the same paper, the notion of a \emph{generalized B-group}, i.e. such finite group whose all proper central $S$-rings are imprimitive, was introduced. It generalizes the classical notion of a \emph{B-group} (\emph{Burnside group}) (see~\cite[Section~25]{Wi}). Central $S$-rings over projective special linear groups were studied in~\cite{HumW}. In fact, some results on the automorphism groups of central $S$-rings over almost simple groups were obtained in~\cite{GGRV,PV2}.

One of the crucial properties of $S$-rings is the \emph{schurity} property. Recall that an $S$-ring over the group $G$ is called \emph{schurian} if it is the transitivity module of a point stabilizer in a subgroup of the symmetric group $\sym(G)$ that contains all right translations of $G$. The schurity property for $S$-rings is closely related to the Cayley graph isomorphism problem (see~\cite[Section~2.2.3]{CP} and~\cite{KP}). Schur had conjectured, in our terms, that every $S$-ring is schurian. However, this conjecture was disproved by Wielandt~\cite{Wi}. In honor of this Schur's fallacy, P\"{o}schel suggested the following definition~\cite{Po}. The group $G$ is said to be \emph{Schur} if every $S$-ring over $G$ is schurian. The problem of determining of all Schur groups suggested by P\"{o}schel was studied during more than fourty years. The most of the results on this problem  are concerned with abelian groups (see~\cite{EKP1,EKP2,MP2,PR,Ry1,Ry2}).

Necessary conditions of schurity for nonabelian groups were obtained in~\cite{PV}. Computer calculations~\cite{Ziv} imply that there are nonabelian Schur groups. On the other hand, it is not known up to present moment whether there exists an infinite family of nonabelian Schur groups. One of the reasons for this is that the attempts to find such family meet some hard unsolved problems, e.g. the problem of existence of a difference set in a cyclic group (see~\cite{MP2}). By the above reasons, the problem of determining of all nonabelian Schur groups seems to be hopeless. 

The main goal of the present paper is to study the schurity property of central $S$-rings over nonabelian groups. Following the strategy from~\cite{CMP}, we introduce the next definition. The group~$G$ is defined to be \emph{generalized Schur} if every central $S$-ring over~$G$ is schurian. Observe that every Schur group is generalized Schur and in abelian case schurity and generalized schurity are equivalent. As it will be shown further, there are infinitely many generalized Schur groups which are not Schur. We suggest the following problem which is a modification of the P\"{o}schel problem. 

\begin{prob}
 \emph{Determine all generalized Schur groups.} 
\end{prob}

In this paper, we establish basic properties of generalized Schur groups and found infinite families of such groups. The first result on Schur groups is the P\"{o}schel theorem which states that cyclic $p$-groups are Schur and if $p\geq 5$ is a prime, then a Schur $p$-group is cyclic. Our first result is concerned with generalized Schur $p$-groups. 

\begin{theo}\label{main1}
Let $p$ be a prime.
\\
\\
\noindent $(1)$ If a noncyclic $p$-group is generalized Schur, then $p\in\{2,3\}$.
\\
\\
\noindent $(2)$ If $p\in \{2,3\}$, then a $p$-group with a maximal cyclic subgroup is generalized Schur.

\end{theo}
\noindent Statement~$2$ implies, in particular, that there are infinitely many nonabelian generalized Schur groups. Observe that nonabelian $2$- and $3$-groups containing a maximal cyclic subgroup except the quaternion group and, possibly, dihedral $2$-groups are not Schur by~\cite[Theorem~11.1]{MP2},~\cite[Theorem~4.2]{PV}, and~\cite[Theorem~1.1]{Ry1}. 

The second result of the present paper is concerned with Camina groups. Recall that $G$ is a \emph{Camina} group if $G$ has a proper nontrivial normal subgroup $H$ such that each $H$-coset distinct from~$H$ is contained in a conjugacy class of~$G$~\cite{Camina}. In this case, the pair $(G,H)$ is called a \emph{Camina pair}. The class of the Camina groups includes, in particular, all Frobenius and extraspecial groups. In~\cite{CMP}, it was proved that every Camina group is a generalized B-group. We prove a sufficient condition of generalized schurity for Camina groups.

\begin{theo}\label{main2}
Let $(G,H)$ be a Camina pair. If $H$ and $G/H$ are generalized Schur groups, then so is $G$. In particular, a Frobenius group with generalized Schur kernel and complement is generalized Schur. 
\end{theo}

Note that the above condition is not necessary. If $(G,H)$ is a Camina pair and $G$ is generalized Schur, then $H$ can be non-generalized Schur in general (see Example~$3$). Since groups of prime orders are Schur~\cite{Po} and a nonabelian group whose order is a product of two primes is Frobenius, the following statement immediately follows from Theorem~\ref{main1}.

\begin{corl}\label{pq}
Let $p$ and $q$ are primes such that $q \equiv 1 \pmod p$. Then the nonabelian group of order~$pq$ is generalized Schur.
\end{corl}
\noindent Recall that schurity of cyclic groups whose order is a product of two primes was verified in~\cite{KP}; this was one of the first results on Schur groups. There are infinitely many nonabelian non-Schur groups of order~$pq$, where $p$ and $q$ are primes, by~\cite[Corollary~5.3]{PV}.

Finally, we obtain the classification of dihedral generalized Schur groups.

\begin{theo}\label{main3}
Let $n\geq 1$ be an integer. The dihedral group of order~$2n$ is generalized Schur if and only if $n$ belongs to one of the following families of integers:
$$p^k, pq^k, 2pq^k, pqr, 2pqr,$$
where $p,q,r$ are primes and $k\geq 0$ is an integer.
\end{theo}

In fact, we prove that the dihedral group of order~$2n$ is generalized Schur if and only if the cyclic group of order~$n$ is Schur (Proposition~\ref{criterion}). Modulo the classification of cyclic Schur groups~\cite{EKP1}, this gives the required. Note that there are infinitely many dihedral groups from Theorem~\ref{main3} which are not Schur (see~\cite[Section~5]{PV}).

We finish the introduction with the brief outline of the paper. Section~$2$ contains a necessary background of $S$-rings. In Section~$3$, we establish basic properties of generalized Schur groups. In Section~$4$, $5$, and $6$, we prove Theorems~\ref{main1},~\ref{main2}, and~\ref{main3}, respectively.

\section{$S$-rings}

In this section, we provide a necessary background of $S$-rings. The material of the section can be found, e.g., in~\cite{MP2}. 

\subsection{Definitions and basic facts}
Let $G$ be a finite group and $\mathbb{Z}G$  the integer group ring. The identity element and the set of all nonidentity elements of $G$ are denoted by~$e$ and~$G^\#$, respectively. The symmetric group of the set $G$ is denoted by~$\sym(G)$. If $K\leq \sym(G)$, then the set of all orbits of $K$ on $G$ is denoted by $\orb(K,G)$. The subgroups of $\sym(G)$ induced by the right multiplications and by the left multiplications of $G$ are denoted by $G_{r}$ and $G_l$, respectively. If $X\subseteq G$, then the element $\sum \limits_{x\in X} {x}$ of the group ring $\mathbb{Z}G$ is denoted by~$\underline{X}$. The set $\{x^{-1}:x\in X\}$ is denoted by $X^{-1}$.

A subring  $\mathcal{A}\subseteq \mathbb{Z} G$ is called an \emph{$S$-ring} (a \emph{Schur ring}) over $G$ if there exists a partition $\mathcal{S}=\mathcal{S}(\mathcal{A})$ of~$G$ such that:

$(1)$ $\{e\}\in\mathcal{S}$;

$(2)$  if $X\in\mathcal{S}$, then $X^{-1}\in\mathcal{S}$;

$(3)$ $\mathcal{A}=\Span_{\mathbb{Z}}\{\underline{X}:\ X\in\mathcal{S}\}$.

\noindent The theory of $S$-rings was initiated by Schur within the studying of permutation groups~\cite{Schur} and later it was developed by Wielandt~\cite{Wi}. The elements of $\mathcal{S}$ are called the \emph{basic sets} of  $\mathcal{A}$ and the number $\rk(\mathcal{A})=|\mathcal{S}|$ is called the \emph{rank} of~$\mathcal{A}$. The multiset consisting of sizes of all basic sets of $\mathcal{A}$ in that multiplicity of each element~$k$ is equal to the number of basic sets of size~$k$ is denoted by $\mathcal{N}(\mathcal{A})$. 

Clearly, $\mathbb{Z}G$ is an $S$-ring over $G$. Another obvious example of an $S$-ring is defined by the partition $\{\{e\},G^\#\}$. This $S$-ring is called \emph{trivial} and denoted by $\mathcal{T}_G$. 

Let $\mathcal{A}^\prime$ be an $S$-ring over $G$. One can see that $\mathcal{A}^\prime\leq \mathcal{A}$, i.e. $\mathcal{A}^\prime$ is a subring of $\mathcal{A}$, if and only if every basic set of $\mathcal{A}^\prime$ is a union of some basic sets of $\mathcal{A}$. 

A set $X \subseteq G$ is called an \emph{$\mathcal{A}$-set} if $\underline{X}\in \mathcal{A}$. A subgroup $H \leq G$ is called an \emph{$\mathcal{A}$-subgroup} if $H$ is an $\mathcal{A}$-set. One can check that for every $\mathcal{A}$-set $X$, the groups $\langle X \rangle$ and $\rad(X)=\{g\in G:\ gX=Xg=X\}$ are $\mathcal{A}$-subgroups. The $S$-ring $\mathcal{A}$ is called \emph{primitive} if there are no nontrivial proper $\mathcal{A}$-subgroups of $G$ and \emph{imprimitive} otherwise.

\begin{lemm}\cite[Theorem~2.6]{MP2} \label{separat}
Let $\mathcal{A}$ be an $S$-ring over a group $G$. Suppose that $X\in \mathcal{S}(\mathcal{A})$ and $H\leq G$ are such that
$$X\cap H \neq \varnothing,~X \setminus H \neq \varnothing,~\text{and}~\langle X\cap H\rangle \leq \rad(X\setminus H).$$
Then $X=\langle X \rangle \setminus \rad(X)$ and $\rad(X) \leq H$.
\end{lemm}

Let $\{e\}\leq L \unlhd U\leq G$. A section $U/L$ is called an \emph{$\mathcal{A}$-section} if $U$ and $L$ are $\mathcal{A}$-subgroups. If $S=U/L$ is an $\mathcal{A}$-section, then the module
$$\mathcal{A}_S=Span_{\mathbb{Z}}\left\{\underline{X}^{\pi}:~X\in\mathcal{S}(\mathcal{A}),~X\subseteq U\right\},$$
where $\pi:U\rightarrow U/L$ is the canonical epimorphism, is an $S$-ring over $S$.

\subsection{Isomorphisms and schurity}

Let $X\subseteq G$. By the \emph{Cayley digraph} $\cay(G,X)$ over $G$ with connection set $X$, we mean the digraph with vertex set $G$ and arc set $\{(g,xg):~x\in X,~g\in G\}$. Let $\mathcal{A}^\prime$ be an $S$-ring over a group $G^\prime$. A bijection $f$ from $G$ to $G^\prime$ is called an \emph{isomorphism} from $\mathcal{A}$ to $\mathcal{A}^\prime$ if for every $X\in \mathcal{S}(\mathcal{A})$ there exists $X^\prime\in \mathcal{S}(\mathcal{A}^\prime)$ such that $f$ is an isomorphism from $\cay(G,X)$ to $\cay(G^\prime,X^\prime)$. Two $S$-rings are called \emph{isomorphic} if there exists an isomorphism between them.

The \emph{automorphism group} $\aut(\mathcal{A})$ of the $S$-ring $\mathcal{A}$ is defined to be the group 
$$\bigcap \limits_{X\in \mathcal{S}(\mathcal{A})} \aut(\cay(G,X)).$$
Since $\aut(\cay(G,X))\geq G_{r}$ for every $X\in \mathcal{S}(\mathcal{A})$, we conclude that $\aut(\mathcal{A})\geq G_{r}$. One can check that $\mathcal{A}$ is isomorphic to an $S$-ring over a group $H$ if and only if $\aut(\mathcal{A})$ has a regular subgroup isomorphic to~$H$.

Let $K$ be a subgroup of $\sym(G)$ containing $G_{r}$. Schur proved in~\cite{Schur} that the $\mathbb{Z}$-submodule
$$V(K,G)=\Span_{\mathbb{Z}}\{\underline{X}:~X\in \orb(K_e,~G)\},$$
where $K_e$ is a stabilizer of $e$ in $K$, is an $S$-ring over $G$. An $S$-ring $\mathcal{A}$ over  $G$ is called \emph{schurian} if $\mathcal{A}=V(K,G)$ for some $K\leq \sym(G)$ with $K\geq G_{r}$. Clearly, $\mathcal{T}_G=V(\sym(G),G)$ and hence $\mathcal{T}_G$ is schurian.

Let $K \leq \aut(G)$. Then $\orb(K,G)$ forms a partition of $G$ that defines the $S$-ring $\mathcal{A}$ over~$G$. In this case, $\mathcal{A}$ is called \emph{cyclotomic} and denoted by $\cyc(K,G)$. If $\mathcal{A}=\cyc(K,G)$ for some $K\leq \aut(G)$, then $\mathcal{A}=V(G_rK,G)$. So every cyclotomic $S$-ring is schurian.

The group $G$ is called a \emph{Schur} group if every $S$-ring over $G$ is schurian. It can be checked easily that a section of Schur group is Schur. Every group of order at most~$15$ is Schur (see~\cite{Ziv}).

\subsection{Wreath products}

Let $U/L$ be an $\mathcal{A}$-section of $G$. The $S$-ring~$\mathcal{A}$ is called the \emph{$U/L$-wreath product} or \emph{generalized wreath product} of $\mathcal{A}_U$ and $\mathcal{A}_{G/L}$ if $L$ is normal in $G$ and $L\leq\rad(X)$ for each basic set $X$ outside~$U$. If $L>\{e\}$ and $U<G$, then the $U/L$-wreath product is called \emph{nontrivial}. The notion of the generalized wreath product of $S$-rings was introduced in~\cite{EP}. If $U=L$, then the $U/L$-wreath product coincides with the usual \emph{wreath product} $\mathcal{A}_L\wr \mathcal{A}_{G/L}$ of $\mathcal{A}_L$ and $\mathcal{A}_{G/L}$. The following criterion of schurity for the wreath product can be found, e.g., in~\cite[Theorem~3.2]{MP2}.

\begin{lemm}\label{schurwreath}
Let $\mathcal{A}$ be an $S$-ring over a group~$G$ such that $\mathcal{A}=\mathcal{A}_H\wr \mathcal{A}_{G/H}$ for some normal $\mathcal{A}$-subgroup~$H$. Then $\mathcal{A}$ is schurian if and only if $\mathcal{A}_H$ and $\mathcal{A}_{G/H}$ so are. 
\end{lemm}

\begin{lemm}\label{schurgenwreath}
Let $\mathcal{A}$ be an $S$-ring over an abelian group~$G$. Suppose that $\mathcal{A}$ is the nontrivial $S$-wreath product for some $\mathcal{A}$-section $S=U/L$, the $S$-rings $\mathcal{A}_U$ and $\mathcal{A}_{G/L}$ are schurian, and $|S|\leq 2$. Then $\mathcal{A}$ is schurian. 
\end{lemm}

\begin{proof}
The group $\aut(\mathcal{A}_S)$ is trivial or isomorphic to~$\sym(2)$. So $\aut(\mathcal{A}_S)$ is $2$-isolated (see~\cite{MP2} for the definition). Therefore $\mathcal{A}$ is schurian by~\cite[Corollary~10.3]{MP2}.
\end{proof}

\begin{lemm}\label{rank3}
An $S$-ring of rank~$3$ over a dihedral group is schurian.
\end{lemm}

\begin{proof}
Let $\mathcal{A}$ be an $S$-ring of rank~$3$ over a dihedral group $G$. The group $G$ is a strong B-group (see~\cite[Theorem~25.6 (i)]{Wi}), i.e. only the trivial $S$-ring over $G$ is primitive. So $\mathcal{A}$ is imprimitive and hence there exists a proper nontrivial $\mathcal{A}$-subgroup $H$. Since $\rk(\mathcal{A})=3$, the basic sets of $\mathcal{A}$ are $\{e\}$, $H^\#$, and $G\setminus H$. Therefore $\mathcal{A}$ is isomorphic to the wreath product of $\mathcal{T}_H$ and $\mathcal{T}_{G/H}$ by~\cite[Corollary~3.3]{MP2}. Thus, $\mathcal{A}$ is schurian by Lemma~\ref{schurwreath}.
\end{proof}

\section{Central $S$-rings and generalized Schur groups}

A subset $X$ of $G$ is called \emph{normal} if $X$ is closed under conjugations by the elements of $G$ or, equivalently, if $X$ is a union of some conjugacy classes of $G$. The center $Z(\mathbb{Z}G)$ of the group ring $\mathbb{Z}G$ is an $S$-ring over $G$. The basic sets of $Z(\mathbb{Z}G)$ are exactly conjugacy classes of $G$ and $Z(\mathbb{Z}G)=V(G_r\Inn(G),G)=\cyc(\Inn(G),G)$. If $G$ is abelian, then $Z(\mathbb{Z}G)=\mathbb{Z}G$.

Following~\cite{CMP}, we say that an $S$-ring $\mathcal{A}$ over $G$ is \emph{central} if $\mathcal{A}\leq Z(\mathbb{Z}G)$. If $\mathcal{A}$ is central, then $\aut(\mathcal{A})\geq G_l$. The group $G$ is defined to be \emph{generalized Schur} if every central $S$-ring over $G$ is schurian. Clearly, if $G$ is abelian, then every $S$-ring over $G$ is central and hence schurity and generalized schurity are equivalent. In this section, we provide some basic properties of central $S$-rings and generalized Schur groups.

\begin{lemm}\label{normalsubgroup}
Let $\mathcal{A}$ be a central $S$-ring over a group $G$. Then every $\mathcal{A}$-subgroup of $G$ is normal.
\end{lemm}

\begin{proof}
Since $\mathcal{A}$ is central, every $\mathcal{A}$-subgroup is a union of some conjugacy classes of $G$ and hence normal.
\end{proof}

If $m$ is a positive integer and $X\subseteq G$, then put $X^{(m)}=\{x^m:~x\in X\}$. The following statement is a generalization of the first Schur theorem on multipliers~\cite[Theorem~23.9, (a)]{Wi}.

\begin{lemm}\cite[Theorem~1.2]{CMP}\label{centralburn}
Let $\mathcal{A}$ be a central $S$-ring over a group $G$ and $m$ an integer coprime to~$|G|$. Then $X^{(m)}\in \mathcal{S}(\mathcal{A})$ for every $X\in\mathcal{S}(\mathcal{A})$.
\end{lemm}

Let $G$ be a group and $A$ and $B$ are subgroups of $G$. The group $G$ is defined to be an (\emph{inner}) \emph{partial semidirect product} of $A$ and $B$ if $A$ is normal in $G$ and $G=AB$. The group $G$ is defined to be an \emph{(inner) central product} of $A$ and $B$ if every element of $A$ commutes with every element of $B$ and $G=AB$ (see~\cite[p.~25]{Gor} for the definitions).

\begin{lemm}\label{semidirect}
Every central $S$-ring over a partial semidirect product of two abelian groups is isomorphic to an $S$-ring over the central product of these groups.
\end{lemm}

\begin{proof}
Let $\mathcal{A}$ be a central $S$-ring over a group $G$. Suppose that $G$ is a partial semidirect product of its abelian subgroups $A$ and $B$. Clearly, every element of $A_l$ commutes with every element of $B_r$. So $K=A_lB_r$ is a central product of $A_l$ and $B_r$. Since $\mathcal{A}$ is central, $K\leq \aut(\mathcal{A})$. Observe that $K$ is transitive on $G$ because $G=AB$. Since $A$ and $B$ are abelian, $K$ is also abelian. Therefore $K$ is regular. Thus, $\aut(\mathcal{A})$ has a regular subgroup isomorphic to the central product of $A$ and $B$ and hence $\mathcal{A}$ is isomorphic to an $S$-ring over the central product of $A$ and $B$. 
\end{proof}

Lemma~\ref{semidirect} implies, in particular, that every central $S$-ring over a semidirect product of two abelian groups is isomorphic to an $S$-ring over the direct product of them.

If $H$ is a subgroup of $G$ and $h\in H$, then the conjugacy class of $H$ containing~$h$ is denoted by~$\Cla(h,H)$.

\begin{lemm}\label{quotient}
Let $G$ be a generalized Schur group and $H$ a normal subgroup of~$G$. Then $G/H$ is generalized Schur and $H$ is generalized Schur whenever $\Cla(h,H)=\Cla(h,G)$ for every $h\in H$.
\end{lemm}

\begin{proof}
At first, let us prove that $G/H$ is generalized Schur. Assume the contrary. Then there exists a nonschurian central $S$-ring $\mathcal{B}$ over $G/H$. Put $\mathcal{A}=\mathcal{T}_H\wr \mathcal{B}$. From Lemma~\ref{schurwreath} it follows that $\mathcal{A}$ is nonschurian. Further we will prove that every $X\in\mathcal{S}(\mathcal{A})$ is normal and hence $\mathcal{A}$ is central. The required is clear if $X=\{e\}$ or $X=H^\#$. Suppose that $X\subseteq G\setminus H$. Let $x\in X$ and $g\in G$. Then 
$$g^{-1}xg\in Hg^{-1}xg=Hg^{-1}HxHg\in X^\pi,$$
where $\pi:G\rightarrow G/H$ is the canonical epimorphism and the latter inclusion holds because $\mathcal{A}_{G/H}\cong \mathcal{B}$ is central. Since $H\leq \rad(X)$, the above equality implies that $g^{-1}xg\in X$. Therefore $X$ is normal. Thus, $\mathcal{A}$ is nonschurian central $S$-ring over $G$, a contradiction to generalized schurity of $G$.

Now let us prove that $H$ is generalized Schur whenever $\Cla(h,H)=\Cla(h,G)$ for every $h\in H$. Assume the contrary. Then there exist a nonschurian central $S$-ring $\mathcal{B}$ over $H$. Put $\mathcal{A}=\mathcal{B}\wr \mathcal{T}_{G/H}$. The $S$-ring $\mathcal{A}$ is nonschurian by Lemma~\ref{schurwreath} and it is central by the condition of the lemma. So we obtain a contradiction to generalized schurity of $G$.
\end{proof}

\begin{rem1}
Note that the condition $\Cla(h,H)=\Cla(h,G)$ for every $h\in H$ in Lemma~\ref{quotient} is essential. In general, a normal subgroup of a generalized Schur group can be not generalized Schur (see Example~$3$).
\end{rem1}

\vspace{5mm}

\noindent \textbf{Example 1.} The alternating group $A_5$ is generalized Schur. 

\vspace{2mm}

\noindent Let us prove that every central $S$-ring $\mathcal{A}$ over a group $G=A_5$ is schurian. Lemma~\ref{normalsubgroup} implies that every $\mathcal{A}$-subgroup is normal in $G$. Since $G$ is simple, $\mathcal{A}$ is primitive, i.e. there are no proper nontrivial $\mathcal{A}$-subgroups. Let $X_{\max}$ be a basic set of $\mathcal{A}$ of maximal size. From~\cite[Theorem~3.1.8 (2)]{CP} it follows that 
\begin{equation}\label{coprime}
\GCD(|X_{\max}|,|X|)>1
\end{equation}
for every $X\in \mathcal{S}(\mathcal{A})$. 

The group $G$ has~$5$ conjugacy classes of sizes~$1$, $12$, $12$, $15$, $20$. Two conjugacy classes of size~$12$ consist of cycles of length~$5$ and its union is a conjugacy class of the symmetric group $S_5$. Conjugacy classes of sizes $15$ and $20$ are also conjugacy classes of $S_5$.  Clearly, $\mathcal{N}(Z(\mathbb{Z}G))=\{1,12,12,15,20\}$ and $\rk(Z(\mathbb{Z}G))=5$. Since $\mathcal{A}\leq Z(\mathbb{Z}G)$, we conclude that $\rk(\mathcal{A})\in \{2,3,4,5\}$. The $S$-ring $\mathcal{A}$ is obviously schurian if $\rk(\mathcal{A})=2$ or $\rk(\mathcal{A})=5$. Suppose that $\rk(\mathcal{A})=3$. Then 
$$\mathcal{N}(\mathcal{A})\in \{\{1,12,47\},\{1,15,44\},\{1,20,39\}\}.$$
In all cases we have a contradiction to Eq.~\eqref{coprime}. 

Now suppose that $\rk(\mathcal{A})=4$. Then an easy enumeration of all possibilities implies that $\mathcal{N}(\mathcal{A})$ coincides with one of the following sets:
$$\{1,12,12,35\},\{1,12,20,27\},\{1,12,15,32\}, \{1,15,20,24\}.$$
In the first three cases we have a contradiction to Eq.~\eqref{coprime}. In the latter case the basic sets of $\mathcal{A}$ are exactly conjugacy classes of $S_5$ inside $G=A_5$. So $\mathcal{A}=\cyc(\aut(G),G)$ and hence $\mathcal{A}$ is schurian.
\vspace{5mm}

Note that all $S$-rings over $A_5$ were enumerated in~\cite{KZ} and it is possible to check that $A_5$ is generalized Schur using this result.

\section{Proof of Theorem~\ref{main1}}

Firstly, let us prove Statement~$1$. Let $G$ be a noncyclic generalized Schur $p$-group for some prime~$p$. Since $G$ is noncyclic, the group $G/\Phi (G)$, where $\Phi(G)$ is the Frattini subgroup of $G$, is an elementary abelian group of rank at least~$2$. Lemma~\ref{quotient} implies that $G/\Phi (G)$ is generalized Schur and hence Schur. However, an elementary abelian $p$-group of rank at least~$2$ is Schur only if $p\in\{2,3\}$ by~\cite{Po}. 

Now let us prove Statement~$2$. Let $p\in\{2,3\}$, $k\geq 1$, and $G$ a group of order $p^k$ with a maximal cyclic subgroup. If $G$ is abelian, then $G$ is cyclic or isomorphic to $C_{p^{k-1}}\times C_p$, where $C_n$ denotes a cyclic group of order~$n$. In the former case, $G$ is Schur by~\cite{Po}, whereas in the latter case by~\cite[Theorem~1.1]{MP2} if $p=2$ and by~\cite[Theorem~1.1]{Ry1} if $p=3$. 

Suppose that $G$ is nonabelian. Then $k\geq 3$ and $G$ is isomorphic to one of the groups from~\cite[p.~193, Theorem~4.4]{Gor}. Each of these groups except the generalized quaternion group $Q_{2^k}$ of order $2^k$ is a semidirect product of two cyclic subgroups of orders $p^{k-1}$ and $p$ and $Q_{2^k}$ is a partial semidirect product of two cyclic subgroups of orders $2^{k-1}$ and $4$ with intersection of order~$2$. It is easy to check that in all cases the central product of the above mentioned cyclic subgroups is isomorphic to $C_p\times C_{p^{k-1}}$. Therefore every central $S$-ring over $G$ is isomorphic to an $S$-ring over $C_{p^{k-1}}\times C_p$ by Lemma~\ref{semidirect} and hence schurian  by~\cite[Theorem~1.1]{MP2} if $p=2$ and by~\cite[Theorem~1.1]{Ry2} if $p=3$. Thus, $G$ is generalized Schur.

\section{Proof of Theorem~\ref{main2}}
 We start with the proposition which describes the structure of central $S$-rings over Camina groups.

\begin{prop}\label{3wreath}
Let $(G,H)$ be a Camina pair and $\mathcal{A}$ a central $S$-ring over $G$. Then 
$$\mathcal{A}=(\mathcal{A}_L\wr \mathcal{T}_{U/L})\wr \mathcal{A}_{G/U}$$
for some normal $\mathcal{A}$-subgroups $L$ and $U$ such that $\{e\}\leq L\leq H \leq U \leq G$.
\end{prop}

\begin{lemm}\label{caminbasic}
In the notations of Proposition~\ref{3wreath}, let $X\in \mathcal{S}(\mathcal{A})$ such that $X_1=X\cap (G\setminus H)\neq \varnothing$. Then $H\leq \rad(X_1)$.
\end{lemm}

\begin{proof}
The set $X_1$ is a union of some conjugacy classes of $G$ because $\mathcal{A}$ is central and $H$ is normal in $G$. Since $G$ is a Camina group, each conjugacy class of~$G$ lying outside $H$ is a union of some $H$-cosets. Therefore $H\leq \rad(X_1)$.
\end{proof}

\begin{proof}[Proof of Proposition~\ref{3wreath}]
At first, suppose that $H$ is an $\mathcal{A}$-subgroup. Then $H\leq \rad(X)$ for every $X\in \mathcal{S}(\mathcal{A})$ outside~$H$ by Lemma~\ref{caminbasic}. Therefore $\mathcal{A}=\mathcal{A}_H\wr\mathcal{A}_{G/H}$ and the proposition holds for $L=U=H$.   

Now suppose that $H$ is not an $\mathcal{A}$-subgroup. Then there exists a basic set $X\in \mathcal{S}(\mathcal{A})$ such that $X_0=X\cap H\neq \varnothing$ and $X_1=X\cap (G\setminus H)\neq \varnothing$. Due to Lemma~\ref{caminbasic}, we have $H\leq \rad(X_1)$. So all conditions of Lemma~\ref{separat} hold for $X$ and $H$. Thus, 
\begin{equation}\label{sep2}
X=U \setminus L~\text{and}~L\leq H,
\end{equation}
where $U=\langle X \rangle$ and $L=\rad(X)$. 

Clearly, $U$ and $L$ are $\mathcal{A}$-subgroups. Moreover, they are normal in~$G$ by Lemma~\ref{normalsubgroup}. Note that $L<H$ because $L$ is an $\mathcal{A}$-subgroup whereas $H$ is not an $\mathcal{A}$-subgroup by the assumption. On the other hand, $H\leq \rad(X_1)\leq U$. Together with Eq.~\eqref{sep2}, this implies that $U$ is the minimal $\mathcal{A}$-subgroup containing~$H$. If $Y\in\mathcal{S}(\mathcal{A})$ lies outside $U$, then $H\leq \rad(Y)$ by Lemma~\ref{caminbasic} and hence $U\leq \rad(Y)$. Therefore
\begin{equation}\label{wr1}
\mathcal{A}=\mathcal{A}_U \wr \mathcal{A}_{G/U}.
\end{equation}
From Eq.~\eqref{sep2} it follows that $\mathcal{A}_U=\mathcal{A}_L\wr \mathcal{T}_{U/L}$. Now Eq.~\eqref{wr1} yields $\mathcal{A}=(\mathcal{A}_L\wr \mathcal{T}_{U/L})\wr \mathcal{A}_{G/U}$ and we are done.
\end{proof}

\begin{proof}[Proof of Theorem~\ref{main2}]
Assume the contrary that $G$ is not a generalized Schur group. Then there exists a nonschurian central $S$-ring $\mathcal{A}$ over~$G$. By Proposition~\ref{3wreath}, we have $\mathcal{A}=(\mathcal{A}_L\wr \mathcal{T}_{U/L})\wr \mathcal{A}_{G/U}$ for some normal $\mathcal{A}$-subgroups $L$ and $U$ such that $\{e\}\leq L\leq H \leq U \leq G$. Since $\mathcal{A}$ is nonschurian, $L>\{e\}$ and $\mathcal{A}_L$ is nonschurian or $U<G$ and $\mathcal{A}_{G/U}$ is nonschurian by Lemma~\ref{schurwreath}. 

Suppose that $L>\{e\}$ and $\mathcal{A}_L$ is nonschurian. Then $\mathcal{A}_L\wr \mathcal{T}_{H/L}$ is a nonschurian $S$-ring over~$H$ by Lemma~\ref{schurwreath}. Every basic set of $\mathcal{A}_L$ is a union of some conjugacy classes of~$G$ which lie inside $H$ because $\mathcal{A}$ is central. So $\mathcal{A}_L\wr \mathcal{T}_{H/L}$ is central and we obtain a contradiction to generalized schurity of $H$.

Now suppose that $U<G$ and $\mathcal{A}_{G/U}$ is nonschurian. Clearly, $\mathcal{A}_{G/U}$ is central. So $G/U$ and hence $(G/H)/(U/H)$ is not generalized Schur. Together with the first part of Lemma~\ref{quotient}, this yields that $G/H$ is not generalized Schur, a contradiction to generalized schurity of $G/H$.
\end{proof}

Corollary~\ref{pq} implies that there exist infinitely many generalized Schur Camina groups. The following example shows that there exist infinitely many Camina groups which are not generalized Schur.

\vspace{5mm}

\noindent \textbf{Example 2.} Let $p$, $q$, $r$, and $s$ be distinct odd primes. By the Dirichlet theorem, there exist infinitely many primes~$n$ of the form $n=kpqrs+1$, where $k$ is a positive integer. Let $G\cong C_n\rtimes \aut(C_n)$. Then $G$ is a Frobenius and hence a Camina group. The group $\aut(C_n)$ is a cyclic group of order $n-1=kpqrs$. From~\cite[Theorem~1.1]{EKP1} it follows that $\aut(C_n)$ is not Schur. Therefore $G$ is not generalized Schur by the first part of Lemma~\ref{quotient}.

\vspace{5mm}

The example below yields that a normal subgroup of a generalized Schur group can be not generalized Schur. This example also shows that the statement reverse to Theorem~\ref{main2} does not hold in general. 

\vspace{5mm}

\noindent \textbf{Example 3.} Let $p$ be a prime, $k\geq 2$, and $P\cong C_p^k$. Suppose that $\sigma \in \aut(P)\cong \GL(k,p)$ is a Singer cycle. Put $G=P\rtimes \langle \sigma \rangle$. Since $\sigma$ is a Singer cycle, $P^\#$ is a conjugacy class of~$G$. Observe that $G$ is a Frobenius whose Frobenius kernel is~$P$ and hence $(G,P)$ is a Camina pair. This implies that each conjugacy class of~$G$ outside $P$ is a union of some $P$-cosets. The group $G/P$ is cyclic and hence each conjugacy class of~$G$ outside~$P$ is exactly a $P$-coset. 

Let $\mathcal{A}$ be a central $S$-ring over~$G$. Since $P^\#$ is a conjugacy class of~$G$, we have $\mathcal{A}=\mathcal{T}_P\wr \mathcal{A}_{G/P}$ by Proposition~\ref{3wreath}. The group $G/P$ is a cyclic group of order $p^k-1$. If
\begin{equation}\label{prime}
p^k\in\{2^6,2^7,2^8,2^9,2^{10},2^{11},3^4,3^5,5^2,5^3,7^2,7^3\},
\end{equation}
then $G/P$ is a Schur group by~\cite[Theorem~1.1]{EKP1}. Therefore $\mathcal{A}_{G/P}$ is schurian. In this case, the $S$-ring $\mathcal{A}$ is schurian by Lemma~\ref{schurwreath}. Thus, $G$ is a generalized Schur group whenever Eq.~\eqref{prime} holds. On the other hand, if Eq.~\eqref{prime} holds, then $P$ is not Schur and hence not generalized Schur by~\cite[Theorem~1.2]{EKP2}.

\section{Proof of Theorem~\ref{main3}}

Theorem~\ref{main3} immediately follows from~\cite[Theorem~1.1]{EKP1} and the following proposition.

\begin{prop}\label{criterion}
Let $n\geq 1$ be an integer. The dihedral group of order~$2n$ is generalized Schur if and only if the cyclic group of order~$n$ is Schur.
\end{prop}

\begin{proof}
Let $n\geq 3$ and $G=\langle a, b:~a^n=b^2=e,~a^b=a^{-1}\rangle$. Put $A=\langle a \rangle$ and $B=\langle b \rangle$. Clearly, $G=A\rtimes B$, $A\cong C_n$, $B\cong C_2$, and $G$ is dihedral of order~$2n$. These notations are valid until the end of the section. The conjugacy classes of $G$ inside~$A$ are of the following form~$\{x,x^{-1}\}$, $x\in A$. 

At first, let us prove the ``only if'' part of the proposition. Suppose that $A\cong C_n$ is not Schur. Then due to~\cite[Theorem~1.1, Lemma~2.2]{EKP1}, the integer~$n$ satisfies the conditions of~\cite[Theorem~2.1]{EKP1}. So there exists a nonschurian $S$-ring $\mathcal{A}$ over~$A$ described in~\cite[pp.5-8, Theorem~2.1]{EKP1} such that $X=X^{-1}$ for every $X\in \mathcal{S}(\mathcal{A})$. Due to the latter property of $\mathcal{A}$, the $S$-ring $\mathcal{A}\wr \mathcal{T}_{G/A}$ is central. Since $\mathcal{A}$ is nonschurian, $\mathcal{A}\wr \mathcal{T}_{G/A}$ is nonschurian by Lemma~\ref{schurwreath}. Therefore, $G$ is not generalized Schur.

Now let us prove the ``if'' part of the theorem. If $n$ is odd, then $(G,A)$ is a Camina pair. Since $A$ and $G/A \cong C_2$ are Schur, $G$ is generalized Schur by Theorem~\ref{main2}. Further we assume that $n$ is even. Then the conjugacy classes of~$G$ outside~$A$ are $bA_1$ and $baA_1$, where $A_1$ is a subgroup of~$A$ of index~$2$.  

Let us prove that every central $S$-ring $\mathcal{A}$ over $G$ is schurian. Let $L$ be the maximal $\mathcal{A}$-subgroup of~$A$. 

\begin{lemm}\label{2wreath}
In the above notations, $\mathcal{A}=\mathcal{A}_L\wr \mathcal{A}_{G/L}$, where $\rk(\mathcal{A}_{G/L})\leq 3$, or $\mathcal{A}$ is the $A/A_1$-wreath product.
\end{lemm} 

\begin{proof}
Every basic set of $\mathcal{A}$ outside $L$ contains an element from $bA$ and hence all elements from at least one of the sets $bA_1$, $baA_1$. Therefore there are at most two basic sets of $\mathcal{A}$ outside $L$. If there is exactly one basic set $X$ of $\mathcal{A}$ outside $L$, then $X=G\setminus L$ and hence $\mathcal{A}=\mathcal{A}_L\wr\mathcal{T}_{G/L}$. 

Suppose that there are two distinct basic sets outside $L$, namely $X$ and $Y$. Clearly, $X^\pi,Y^\pi\in\mathcal{S}(\mathcal{A}_{G/L})$ and hence $X^\pi\cap Y^\pi=\varnothing$ or $X^\pi=Y^\pi$, where $\pi$ is the canonical epimorphism from $G$ to $G/L$. In the former case, $L\leq \rad(X)\cap\rad(Y)$ and, consequently, $\mathcal{A}=\mathcal{A}_L\wr \mathcal{A}_{G/L}$ and $\rk(\mathcal{A}_{G/L})=3$. 

In the latter case, $\rk(\mathcal{A}_{G/L})=2$. If $L=A$, then $\{X,Y\}=\{bA_1,baA_1\}$, $A_1=\rad(bA_1)$ is an $\mathcal{A}$-subgroup, and $\mathcal{A}$ is the $A/A_1$-wreath product. Suppose that $L<A$. Since $X^\pi=Y^\pi$, we conclude that both $X$ and $Y$ contain generators of $A$, say $x$ and $y$, respectively. It is easy to see that there exists a positive integer $m$ coprime to $|G|$ such that $y=x^m$. The set $X^{(m)}$ is a basic set of $\mathcal{A}$ by Lemma~\ref{centralburn} and $y\in X^{(m)}$. So $X^{(m)}=Y$. On the other hand, $X$ contain an element~$x^{\prime}$ of order~$2$ from $bA$. One can see that 
$$(x^{\prime})^m=x^{\prime}\in X^{(m)}\cap X=Y\cap X=\varnothing,$$
a contradiction.  
\end{proof}

At first, suppose that $\mathcal{A}=\mathcal{A}_L\wr \mathcal{A}_{G/L}$, where $\rk(\mathcal{A}_{G/L})\leq 3$. The $S$-ring $\mathcal{A}_L$ is schurian because $L\leq A$ and $A$ is Schur by the assumption of the proposition. The $S$-ring $\mathcal{A}_{G/L}$ is obviously schurian if $\rk(\mathcal{A}_{G/L})=2$ or $|G/L|=4$, and $\mathcal{A}_{G/L}$ is schurian by Lemma~\ref{rank3} otherwise. Thus, $\mathcal{A}$ is schurian by Lemma~\ref{schurwreath}.

Now suppose that $\mathcal{A}$ is the $A/A_1$-wreath product. The $S$-ring $\mathcal{A}_A$ is schurian by the assumption of the proposition. The $S$-ring $\mathcal{A}_{G/A_1}$ is schurian because $|G/A_1|=4$. From Lemma~\ref{semidirect} it follows that $\mathcal{A}$ is isomorphic to an $S$-ring over $C_n\times C_2$. Thus, $\mathcal{A}$ is schurian by Lemma~\ref{schurgenwreath}. 
\end{proof}

\begin{rem2}
If $n$ is odd and $C_n$ is a cyclic Schur group, then $C_n\times C_2$ is also Schur (see~\cite[Theorem~1.1]{EKP1}). Therefore Theorem~\ref{main3} for odd~$n$ also follows from Lemma~\ref{semidirect}.
\end{rem2}

\end{document}